\RequirePackage{fix-cm}
\documentclass[smallextended]{svjour3}       
\smartqed  
\usepackage{amssymb,amsmath,multirow}
\usepackage{xcolor}

\newtheorem{thm}{Theorem}
\newtheorem{prop}{Proposition}
\newtheorem{lem}{Lemma}
\newtheorem{cor}{Corollary}
\newtheorem{rem}{Remark}
\newtheorem{ass}{Assumption}
\newtheorem{exam}{Example}
\def\pn{\par\smallskip\noindent}
\def\it proof{\pn {Proof.} }

\begin{document}
\title{Chebyshev Center of the Intersection of Balls: Complexity, Relaxation and Approximation
\thanks{This research was supported  by National Science Fund for Excellent Young Scholars 11822103 and
by NSFC under grants  11571029, 11801173 and  11771056.}
}

\titlerunning{ Chebyshev Center of the Intersection of Balls}        

\author{Yong Xia \and Meijia Yang \and Shu Wang\footnote{Corresponding author} }


\institute{
Y. Xia \and M. Yang \at
             LMIB of the Ministry of Education,
              School of
Mathematics and System Sciences, Beihang University, Beijing,
100191, P. R. China
              \email{yxia@buaa.edu.cn (Y. Xia); 15201644541@163.com (M. Yang)}
\and S. Wang \at
College of Science, North China Institute of Science and Technology, Hebei,
065201, P. R. China
\email{wangshu.0130@163.com}
 }

\date{Received: date / Accepted: date}

\maketitle

\begin{abstract}
We study the $n$-dimensional problem of finding the smallest ball enclosing the intersection of $p$ given balls, the so-called Chebyshev center problem (${\rm CC_B}$). It is a minimax optimization problem and the inner maximization is a uniform quadratic optimization problem (${\rm UQ}$).
 When $p\le n$,  (${\rm UQ}$) is known to enjoy a strong duality and consequently (${\rm CC_B}$) is solved via a standard convex quadratic programming (${\rm SQP}$).
 In this paper, we first prove that (${\rm CC_B}$) is  NP-hard and the special case when $n=2$ is  strongly polynomially solvable.
 With the help of a newly introduced linear programming relaxation (LP),  the (${\rm SQP}$) relaxation is reobtained more directly and the first
approximation bound for the solution obtained by (${\rm SQP}$)  is established for the hard case $p>n$. Finally, also based on (LP), we show that (${\rm CC_B}$) is polynomially solvable  when either $n$ or $p-n(>0)$ is fixed.

\end{abstract}

\section{Introduction}
The problem of finding the circle of smallest radius enclosing a given finite set of points in the plane was introduced by Sylvester in 1857 \cite{SJJ1857}. It is equivalent to detecting the Chebyshev center of the convex hull of the given points. This problem (and the corresponding problem in Euclidean space of any constant dimension) can be solved in linear time \cite{M,W91}. There are some other easy-to-solve Chebyshev center problems. For example, when the set is polyhedral and the estimation error is $l_\infty$ norm, the Chebyshev center can be solved by a linear programming problem \cite{MT85}. When the set is the intersection of two ellipsoids and the problem is defined on the complex domain, the Chebyshev center problem can be recast by a semidefinite programming \cite{BEl07} based on strong Lagrangian duality \cite{BEl06}. Generally,  finding the exact Chebyshev center of a convex set (e.g., the intersection of multiple ellipsoids \cite{EBT}) is a hard problem.

Chebyshev centering problem has some applications in optimization. It serves as a key subproblem in each iteration for solving constrained optimization and equilibrium problems \cite{BBY,NP}. The other example is that the problem of finding the optimal $1$-network\footnote{Generally, the optimal $N$-network of a compact set $C\subset \Bbb R^2$ is to solve the optimization problem $\inf_{S\in \sum_N}h(C,S)$, where $\sum_N$ is the set of all
nonempty sets containing at most $N$ points in $\Bbb R^2$ and $h(C,S)=\sup_{x\in C}\inf_{y\in S}\|x-y\|$ is the Hausdorff deviation of a set $C$ from the other set $S$ \cite{LU}.} of a compact set $C\subseteq \Bbb R^2$ can be reduced to finding the Chebyshev center and Chebyshev radius for $C$ \cite{LU}.

In this paper, we study the problem of finding the Chebyshev center of the intersection of balls. Mathematically, it can be reformulated as
\begin{equation}
{\rm (CC_B)}~~\min_{z}\max_{x\in \Omega}\|x-z\|^2, \label{ccb}
\end{equation}
where $\Omega=\{x\in \Bbb R^n: \|x-a_i\|^2\leq r_i^2,~i=1,\ldots, p\}$, $r_i\in(0,+\infty)$  and $\|\cdot\|$ denotes the $l_2$ norm.  Throughout the paper, we assume that $\Omega$ contains a nonempty interior. Geometrically, $({\rm CC_B})$ is to find the smallest ball (centering at $z$)
enclosing $\Omega$.

The problem (${\rm CC_B}$) was first introduced by Beck \cite{Be07}, where it was shown to be equivalent to a standard quadratic programming reformulation (and hence globally solved in polynomial time) when $p\leq n-1$. Two years later, this sufficient assumption is relaxed to $p\leq n$ by the same author \cite{Be}.

We summarize here some (new) applications of (${\rm CC_B}$):
\begin{exam}
Robust estimation.  Suppose $y_k$, $k=1,\ldots,p$ are $p$ measurements of  the unknown $x$ with bounded noises, i.e., $\|y_k-x\| \leq \rho$ for some $\rho>0$. Then, a robust recover of $x$ can be obtained by solving ${\rm (CC_B)}$.
\end{exam}

\begin{exam}[\cite{GWSR,GSWR}]\label{ex22}
In non-cooperative wireless network positioning, the region for a target to
communicate with some reference nodes is an intersection of many balls. The position error is usually bounded by the diameter of the smallest ball enclosing this region. This leads to a  direct application of ${\rm (CC_B)}$ with $n=2$ or $3$.
\end{exam}

\begin{exam}
Recently, based on  ${\rm (CC_B)}$ with $p=2$, Bubeck et al. \cite{BLS:Nes} proposed an alternative to Nesterov's accelerated gradient descent method for minimizing a smooth and strongly convex objective function.
An open problem was raised there whether one can develop an efficient algorithm based on both limited-memory BFGS (which maintains a history of the past $p$ iterations) and the corresponding ${\rm (CC_B)}$ (with the same $p$).
\end{exam}

\begin{exam}
The Chebyshev center of the general convex set
\[
\left\{x\in\Bbb R^n: f_i(x)\leq 0, i=1,\ldots,p\right\}
\]
with each $f_i(x)$ being strongly convex with parameter $\mu_i$  can be approximately relaxed to ${\rm (CC_B)}$ with
\[
 \Omega=\left\{x\in\Bbb R^n: f_i(x_0)+\nabla f_i(x_0)^T(x-x_0)+\frac{\mu_i}{2}\|x-x_0\|^2\leq 0 , i=1,\ldots,p \right\},
 \]
 where $x_0\in\Bbb R^n$ is fixed and $\nabla f_i(x_0)$ is the gradient of $f_i(x)$ at $x_0$.
\end{exam}

The difficulty of the minimax optimization problem (${\rm CC_B}$) (\ref{ccb}) remains, even when the optimal center $z$ is known in advance.
Actually, for any given $z$, the inner maximization of (\ref{ccb}) in terms of $x$, belongs to the class of  nonconvex quadratic optimization problem with ellipsoid constraints:
\begin{eqnarray*}
({\rm NC{\text-}EQP}) ~~&\max& f(x):= x^TQx+c^Tx\\
 &{\rm s.t.}&~ \|F_ix-g_i\|\leq 1,~i=1,\ldots,p,
\end{eqnarray*}
where $F_i\in \Bbb R^{m\times n}$ and $g_i\in \Bbb R^{m}$. More precisely, the inner maximization of (\ref{ccb}) is a special case of (NC-EQP) with $Q=F_i\equiv I$. It is called uniform quadratic optimization and
denoted by (UQ).  When $F_i\equiv I$, (NC-EQP) is generally NP-hard  but polynomially solvable when $p$ is a fixed constant number \cite{DBAM2013}.

The Lagrangian dual problem of (NC-EQP) can be recast as a semidefinite programming (SDP) relaxation. For a complete procedure, we refer to Section 3.1 of this paper, where the Lagrangian dual problem of (UQ) is reformulated as a semidefinite programming problem. The primal (SDP) relaxation can be derived by directly lifting $xx^T$ to $Y$ satisfying $Y\succeq xx^T$ (see \cite{H15}), and it provides an efficient approach to approximately (and sometimes globally) solve (NC-EQP). For (UQ) (i.e., (NC-EQP) with $Q=F_i\equiv I$),  Beck \cite{Be07} showed that there is no duality gap between (NC-EQP) and (SDP)  when $p\leq n-1$, which is further relaxed to $p\leq n$ in \cite{Be}. Moreover, suppose the origin $0$ is an interior feasible solution and
Slater's condition also holds for the dual (SDP), as shown by Tseng \cite{Ts03}, one can generate a feasible solution $\tilde{x}$ in polynomial time satisfying
\begin{equation}
v({\rm{SDP}})\geq f(\tilde{x})\geq \frac{(1-\gamma)^2}{ (\sqrt{r}+\gamma)^2}\cdot v({\rm{SDP}}),\label{tsb}
\end{equation}
where $r=p$, $\gamma:=\max_{i=1,\ldots,p}\|g_i\|$ and $v(\cdot)$ denotes the optimal value of the problem $(\cdot)$.  Recently, as a further improvement of Tseng's result, Hsia et al. \cite{H15} showed that $r$ in (\ref{tsb}) can be decreased from $p$ to 
$\min\left\{n+1,\left\lceil\frac{ \sqrt{8p+17}-3}{2}\right\rceil\right\}$, where $\lceil(\cdot)\rceil$ is the smallest integer greater than or equal to $(\cdot)$.

The  Lagrangian dual was used by Beck \cite{Be07} to efficiently approximate (${\rm CC_B}$). First,
replacing the inner maximization problem (UQ) with its Lagrangian dual yields a quadratic semidefinite programming (QSDP) approximation of (${\rm CC_B}$). Then,  (QSDP) is shown to be equivalently reduced to a convex standard quadratic programming problem (SQP), which is to minimize a quadratic objective function over the unit simplex. To our best knowledge, the quality of  the approximate solution provided by (SQP) remains unknown.

In this paper, we provide an in-depth investigation on (UQ) and (${\rm CC_B}$).
Our main contributions are summarized as follows:
\begin{itemize}
\item (${\rm CC_B}$) is NP-hard. The special case in the plane (i.e., $n=2$) is efficiently and strongly polynomially solvable in $O(p^2)$ time. (Section 2)
\item We propose a new linear programming (LP) relaxation for (UQ), which is equivalent to the (SDP) relaxation under the assumption that strong duality does not hold for (UQ). New sufficient conditions (which generalize the condition $p\le n$) are provided with the help of (LP). (Section 3)
\item An approximation bound for the (LP) relaxation of (UQ) is established, which is tighter than the general ratio (\ref{tsb}) and no longer dependent on $p$. (Section 3)
\item The (LP) relaxation of the inner maximization leads to a new and simple derivation of the standard quadratic programming (SQP) approximation of  (${\rm CC_B}$).   (Section 4)
\item  The quality of the solution returned by (SQP) for approximating (${\rm CC_B}$) is analyzed. More precisely, the first approximation bound of (SQP) is established.  (Section 4)
\item Both (${\rm UQ}$) and (${\rm CC_B}$) can be solved in polynomial time when either $n$ or $p-n(>0)$ is a fixed integer. Moreover, under this assumption, (${\rm UQ}$) could be strongly polynomially solvable.  (Section 5)
\end{itemize}

Throughout the paper, $v(\cdot)$ stands for the optimal value of problem $(\cdot)$.
Denote by $Q\succ(\succeq)0$ that $Q$ is positive (semi)definite.
The inner product of two matrices $A$ and $B$ is denoted by $A\bullet B={\rm Tr}(AB^T)=\sum_{i=1}^n\sum_{j=1}^na_{ij}b_{ij}$.  Denote by rank$(A)$ and $\|A\|$ the rank and spectral norm of $A$, respectively.
For any set $\Omega\subseteq\Bbb R^n$, int$(\Omega)$ denotes the set of all the interior points in $\Omega$.
Denote by ${\rm conv}\{\Omega\}$ the convex hull of $\Omega$.   $\mathcal{N}(\cdot)$ is the null space of $(\cdot)$ and
$\dim(\cdot)$ denotes the dimension of the (sub)space $(\cdot)$.

\section{Complexity of (${\rm CC_B}$)}
In this section, we study the computational complexity of (${\rm CC_B}$). We first show that the planar (${\rm CC_B}$) can be solved in $O(p^2)$. However, in general, (${\rm CC_B}$) is NP-hard.

\subsection{Strongly polynomial solvability in the plane}
\begin{thm}
When $n=2$,  the computational complexity for globally solving
${\rm (CC_B)}$ is at most $O(p^2)$.
\end{thm}
\begin{proof}
We first characterize the  boundary of $\Omega$, where at least one of the $p$ constraints in  $\Omega$ is active. For each $i=1,\ldots,p$, the $i$-th part of the boundary of $\Omega$ is  the following circular arc(s)
\begin{eqnarray}
{\rm arc}_i&:=&\{x\in \Bbb R^2: \|x-a_{i}\|= r_{i},~\|x-a_{j}\|\leq r_{j},~ j\neq i\}\nonumber \\
&=&\cap_{j\neq i}\{x\in \Bbb R^2:  \|x-a_{i}\|= r_{i},~\|x-a_{j}\|\leq r_{j} \}\nonumber \\
&=&\cap_{j\neq i}\{x\in \Bbb R^2:  \|x-a_{i}\|= r_{i},~2(a_i-a_j)^Tx\leq \|a_i\|^2-\|a_j\|^2+r_j^2-r_i^2\}.\nonumber
\end{eqnarray}
The last equality implies that  ${\rm arc}_i$ is the intersection of $p-1$ circular arcs, each of which is obtained by solving a system consisting of one quadratic equation and one linear inequality. 
Therefore, ${\rm arc}_i$ is the union of at most $p-1$ arcs (denoted by ${\rm arc}_{i,j}~ (j=1,2,\cdots,p-1)$) and hence
determined by at most $p-1$ pairs of ordered endpoints $\{A_{i,2j-1},A_{i,2j}\}$ $(j=1,2,\cdots,p-1)$ on the circumference of the $i$-th circle. Each endpoint $A_{i,k}$ ($k\in\{1,2,\cdots,2p-2\}$) corresponds to a closed-form solution in terms of $a_i$ and $r_i$ ($i=1,2,\cdots,p$).
The complexity of finding out all the pairs of ordered endpoints $\{A_{i,2j-1},A_{i,2j}\}$  $(i=1,2,\cdots,p; j=1,2,\cdots,p-1)$  is $O(p^2)$.

Every pair of distinct points on a circle determines two arcs. We call the longer one the major arc and the other the minor arc. In the equal case, both are called major arcs.
If there is an index $i$ such that $a_i\in \cup_{j=1}^{p-1}{\rm conv}\{{\rm arc}_{i,j}\}$ (i.e., ${\rm arc}_{i,j}$ is a major arc for some index $j$), the Chebyshev center and radius of $\Omega$ are $a_i$ and $r_i$, respectively, that is, the $i$-th circle is the smallest circle cover. Otherwise, all ${\rm arc}_{i,j}$ are minor arcs and the  Chebyshev radius of $\Omega$ is smaller than or equal to $\min_{i=1,\cdots,p}\{r_i\}$.
Then, we have
\[
\Omega=\cup_{i=1,2,\cdots,p;j=1,2,\cdots,p-1} {\rm conv}\{{\rm arc}_{i,j}\}\cup{\rm conv}\{A\},
\]
where $A:=\{A_{i,k}:~i=1,2,\cdots,p; k=1,2,\cdots,2p-2\}$.

One can show that any circle with a radius less than or equal to $r_i$ covering the chord  $[A_{i,2j-1},A_{i,2j}]$ also covers the convex hull of the minor arc ${\rm arc}_{i,j}$. Otherwise, the radius of the circle will be greater than $r_i$.
Therefore, ${\rm (CC_B)}$ reduces to find the smallest ball enclosing either ${\rm conv}\{A\}$ or just the finite set $A$. Applying Welzl's algorithm \cite{W91}, one can find the smallest ball enclosing $p(p-1)$ points in $O(p^2)$ time. Notice that all the involved calculations are exact.
As a conclusion, the worst case time complexity of ${\rm (CC_B)}$ is $O(p^2)$.
\end{proof}

It is unknown whether the above algorithm for solving the planar ${\rm (CC_B)}$ can be extended to solve the general ${\rm (CC_B)}$ of dimension $n\ge 3$.

\subsection{NP-hardness}
We show that the general (${\rm CC_B}$) is NP-hard, especially when $p\ge 2n+2$.  Consider the partitioning problem (PP), which asks whether the linear equation
\[
a^Tx=0,~x\in \{-1,1\}^n
\]
has a solution for any given integer vector $a$. It is well known that (PP) is NP-hard \cite{GJ}.

Let us construct the following quadratic programming problem with ball constraints:
\begin{eqnarray}
({\rm P_0})~~&\max & x^Tx \label{p0:0}\\
&{\rm s.t.}& x^Tx+x_i\leq 1+n,~i=1,\ldots n, \label{p0:1}\\
&& x^Tx-x_i\leq 1+n,~i=1,\ldots n,\label{p0:2}\\
&& x^Tx-a^Tx\leq n,\label{p0:3}\\
&& x^Tx+a^Tx\leq n.\label{p0:4}
\end{eqnarray}
It is trivial to see that $v({\rm P_0})\leq n$. Moreover,  we have the following result.
\begin{lem}\label{CP}
$v({\rm P_0})= n$ if and only if the solution of ${\rm(PP)}$ exists.
\end{lem}
\begin{proof}
Let $x$ be any feasible solution of $({\rm P_0})$. It follows from (\ref{p0:3})-(\ref{p0:4}) that $x^{T}x\le n$.
Suppose ${\rm(PP)}$ has a solution, denoted by $\widetilde{x}\in\{-1,1\}^n$, then $\widetilde{x}$ is a feasible solution of (\ref{p0:0})-(\ref{p0:4}) and $\widetilde{x}^T\widetilde{x}=n$. It follows that  $v({\rm P_0})=n$.

On the contrary, suppose $v({\rm P_0})=n$ with the optimal solution $x^*$, we have
  $x^{*T}x^*= n$. Then, it follows from (\ref{p0:1})-(\ref{p0:2}) that $-1\le x^*_i\le 1$ for $i=1,\ldots,n$, and therefore $x^*_i\in \{-1,1\}$. According to (\ref{p0:3})-(\ref{p0:4}), we obtain $a^Tx^*=0$. Therefore, $x^*$ is a solution of (PP).
The  proof is complete.
\end{proof}

\begin{thm}
The problem ${\rm (CC_B})$ is NP-hard.
\end{thm}
\begin{proof}
The problem ${\rm(P_0)}$ can be reformulated as a special case of (${\rm CC_B}$):
\begin{equation}
\min_{z}\max_{x\in \Omega_0}\|x-z\|^2,\label{zmax}
\end{equation}
where $\Omega_0$ denotes the feasible region of ${\rm(P_0)}$, i.e., the set of all solutions satisfying  (\ref{p0:1})-(\ref{p0:4}).
Notice that for any $x\in \Omega_0$, we always have $-x\in \Omega_0$. Therefore, for any $z\in\Bbb R^n$, it holds that
\begin{eqnarray*}
\max_{x\in \Omega_0}\|x-z\|^2&=&\max_{x\in \Omega_0}\max\{\|x-z\|^2,\|-x-z\|^2\}\\
&=&\max_{x\in \Omega_0}\{x^Tx+z^Tz+\max\{-2x^Tz,2x^Tz\}\}\\
&\geq&\max_{x\in \Omega_0}x^Tx+z^Tz\\
&\geq&\max_{x\in \Omega_0}x^Tx.
\end{eqnarray*}
It is trivial to verify that the above two inequalities hold as equalities if and only if $z=0$.
Consequently, we have
\begin{equation}
\min_{z}\max_{x\in \Omega_0}\|x-z\|^2=\max_{x\in \Omega_0}x^Tx=v({\rm P_0}).\label{z0}
\end{equation}
According to (\ref{z0}) and Lemma \ref{CP}, the NP-hard problem (PP) can be solved via (\ref{zmax}), a special case of (${\rm CC_B}$). The proof is complete.
\end{proof}

\section{The inner maximization: uniform quadratic optimization}\label{Sec:LP}
We study in this section the inner maximization problem of ($\rm CC_B$).
It is a uniform quadratic optimization (UQ), which is generally reformulated  as follows:
\begin{eqnarray*}
{\rm (UQ)}~~&\max& f_0(x)= x^Tx-2a_0^Tx\\
&{\rm s.t.}& x\in \Omega_1:=\left\{x\in\Bbb R^n:
f_i(x)=x^Tx-2a_i^Tx+b_i\leq 0,~i=1,\cdots,p
\right\}.
\end{eqnarray*}
(UQ) is difficult to solve. It is already NP-hard as it contains $({\rm P_0})$ (\ref{p0:0})-(\ref{p0:4}) as a special case.
Throughout this section, we assume that the Slater's condition holds for (UQ), i.e., int$(\Omega_1)\neq\emptyset$. Without loss of generality, using a vector translation if necessary, we can make the following assumption.
\begin{ass}\label{ass2}
Suppose $0\in {\rm int}(\Omega_1)$, or equivalently, $b_i<0$ for $i=1,\cdots,p$.
\end{ass}

\subsection{${\rm SDP}$ relaxations}
Introducing $p$ Lagrangian multipliers $\lambda_1\ge0,\cdots,\lambda_p\ge0$ yields the Lagrangian function of (UQ):
\begin{eqnarray*}
L(x,\lambda)&=&f_0(x)-\sum_{i=1}^p\lambda_if_i(x)\\
&=&
\left(1-\sum_{i=1}^p \lambda_{i}\right)x^Tx+
2\left(\sum_{i=1}^p \lambda_{i} a_{i}-a_0\right)^Tx-\sum_{i=1}^p\lambda_{i}b_i.
\end{eqnarray*}

The Lagrangian dual problem of (UQ) reads as
\begin{eqnarray*}
{\rm (D)}~~\inf_{\lambda\geq 0}~\left\{d(\lambda):=\sup_{x\in \Bbb R^n}L(x,\lambda)\right\}.
\end{eqnarray*}
By using Shor's relaxation scheme \cite{NS}, the dual problem (D) can be reformulated as the dual semidefinite programming ${\rm (D\text{-}SDP)}$ relaxation for (UQ):
\begin{eqnarray*}
{\rm (D\text{-}SDP)}~~& \inf  & \tau\\
&{\rm s.t.}&
\left(\begin{array}{cc}(\sum_{i=1}^p\lambda_{i}-1)I&a_0-\sum_{i=1}^p\lambda_{i}a_{i}\\a_0^T-\sum_{i=1}^p\lambda_{i}a_{i}^T&\tau+ \sum_{i=1}^p\lambda_{i}b_{i}
\end{array}\right)\succeq 0,\\
&&\tau\in \Bbb R,~\lambda_i\geq 0,~i=1,\ldots,p.
\end{eqnarray*}

It is not difficult to  verify that  the conic dual problem of (D-SDP) leads to the following primal SDP:
\begin{eqnarray}
{\rm (SDP)}~~&\max &  \left(\begin{array}{cc}I&-a_{0}\\-a_{0}^T&0
\end{array}\right)\bullet \left(\begin{array}{cc}Y&x\\x^T&1
\end{array}\right) \label{SDP:0}\\
&{\rm s.t.}& \left(\begin{array}{cc}I&-a_{i}\\-a_{i}^T&b_{i}
\end{array}\right)\bullet \left(\begin{array}{cc}Y&x\\x^T&1
\end{array}\right) \leq 0,~i=1,\ldots,p,\label{SDP:1}\\
&& \left(\begin{array}{cc}Y&x\\x^T&1
\end{array}\right)\succeq 0,\label{SDP:2} 
\end{eqnarray}
which can be more directly obtained by
lifting  $xx^T$ to $Y$ satisfying $Y\succeq xx^T$ (see \cite{Be07}). It is trivial to see that $(x,Y)=(0,\epsilon I)$ with $0<\epsilon <\frac{\min\{-b_i\}}{n}$ is an interior feasible point of (${\rm SDP}$). And, $(\lambda,\tau)$, with any fixed $\lambda$ satisfying $\lambda_i>0$, $\sum_{i=1}^p\lambda_i>1$ and sufficiently large $\tau$, is an interior feasible point of (${\rm D\text{-}SDP}$). That is, Slater's condition holds for both (${\rm SDP}$) and (${\rm D\text{-}SDP}$). Consequently,
strong duality holds for (D-SDP) and (SDP) \cite{LS}, i.e.,
\[
v{\rm (SDP)}=v{\rm (D\text{-}SDP)}=v{\rm (D)}
\]
and both optimal values are attained.

For the tightness of the SDP relaxation,
Beck \cite{Be07} established  the following sufficient  condition to guarantee that $v({\rm UQ})=v({\rm SDP})$.
\begin{thm}[\cite{Be07}]
$v({\rm SDP})=v({\rm UQ})$ as long as $p\leq n-1$.
\end{thm}

Two years later, Beck \cite{Be} further relaxed the above sufficient condition to the following one:
\begin{thm}[\cite{Be}]\label{Beck}
$v({\rm SDP})=v({\rm UQ})$ as long as $p\leq n$.
\end{thm}

Before ending this section, we propose a more general sufficient condition to guarantee the strong duality of (${\rm UQ }$) and leave the proof for the next subsection.

\begin{thm}\label{thm:sdp0}
$v({\rm SDP})=v({\rm UQ})$ under one of the following conditions:
\begin{itemize}
\item[(i)] $a_0\not\in$ {\rm conv}$\left\{a_1,\ldots,a_p\right\}$;\\
\item[(ii)]
$\left\{x\in\Bbb R^n: (a_i-a_0)^Tx\geq 0,i=1,\cdots,p\right\}\neq\left\{0\right\}$.
\end{itemize}
\end{thm}
Theorem  \ref{Beck} is a corollary of Theorem \ref{thm:sdp0}, as shown in the following.
\begin{prop}
If $p\leq n$, then either Case  (i) or Case  (ii) holds.
\end{prop}
\begin{proof}
If $p\leq n$ and $a_0\in$ conv$\left\{a_1,\ldots,a_p\right\}$, then
\[
{\rm rank}\left(\left[a_1-a_0,\cdots,a_p-a_0\right]\right)\leq p-1\leq n-1.
\]
Therefore, $\mathcal{N}\left(\left[a_1-a_0,\cdots,a_p-a_0\right]\right)\ge 1$, which implies that
\[
\left\{x\in\Bbb R^n: (a_i-a_0)^Tx\geq 0,i=1,\cdots,p\right\}\neq\left\{0\right\}.
 \]
 That is, if Case $(i)$ is not true, then Case $(ii)$ holds.
\end{proof}

\subsection{Linear programming relaxation}

Simply introducing a new variable $y\in\Bbb R$ to replace the nonconvex term $x^Tx$, we obtain the following  linear programming relaxation:
\begin{eqnarray}
{\rm (LP)}~~&\max& y-2a_0^Tx\nonumber\\
&{\rm s.t.}& y-2a_i^Tx+b_i\leq 0,~i=1,\ldots,p.\label{LP-constraint}
\end{eqnarray}

We first show that the linear programming bound $v$(LP) is proper.
\begin{lem}\label{lem:1}
v$({\rm LP})<+\infty$ if and only if
$a_0\in$ ${\rm conv}\left\{a_1,\ldots,a_p\right\}$.
\end{lem}
\begin{proof}
The dual problem of the linear programming  (${\rm LP}$) reads as follows:
\begin{eqnarray}
{\rm (DLP)}~~&\min& -\sum_{i=1}^pb_i\lambda_i\nonumber\\
&{\rm s.t.}&\sum_{i=1}^p\lambda_ia_i=a_0,\\
&&\sum_{i=1}^p\lambda_i=1,\lambda_i\geq 0,~i=1,\ldots,p.\nonumber
\end{eqnarray}
Since strong duality holds for linear programming,
we have $v({\rm LP})=v({\rm DLP}$). Therefore, $v({\rm LP})<+\infty$ if and only if $({\rm DLP})$ has a feasible solution, that is, $a_0\in$ conv$\left\{a_1,\ldots,a_p\right\}$. The proof is complete.
\end{proof}

Next, it is interesting to see that the linear programming bound $v$(LP) is as tight as the SDP relaxation  if (SDP) is not tight.
\begin{thm}\label{thm:sdp1}
If $v({\rm SDP})>v({\rm UQ})$, then $v({\rm SDP})=v({\rm LP})<+\infty$.
\end{thm}
\begin{proof}
Since Slater's condition holds for (SDP) and its conic dual (D-SDP), we have  $v({\rm SDP})<+\infty$ and it is attained at an optimal solution  $(x^*, Y^*)$.

The constraint (\ref{SDP:2}) implies that
\begin{equation}
{\rm Tr}(Y^*)=I\bullet Y^*\ge x^{*T}x^*.\label{ine:1}
\end{equation}
Substituting (\ref{ine:1}) into (\ref{SDP:1}) implies that $x^*\in\Omega_1$. If the constraint (\ref{ine:1}) holds as an equality, then $Y^*=x^*x^{*T}$ and hence $x^*$ is the optimal solution of (UQ). It must hold that $v({\rm SDP})=v({\rm UQ})$.
Therefore, under the assumption $v({\rm SDP})>v({\rm UQ})$, the constraint (\ref{ine:1}) is inactive, i.e.,  ${\rm Tr}(Y^*)>x^{*T}x^*$.
Define
\[
\widetilde{Y}=x^*x^{*T}+\frac{{\rm Tr}(Y^*)-x^{*T}x^*}{n}\cdot I.
\]
Then, we have ${\rm Tr}(\widetilde{Y})={\rm Tr}(Y^*)$, $\widetilde{Y}\succ x^*x^{*T}$
and $(x^*,\widetilde{Y})$ remains an optimal solution to $({\rm SDP})$. There is a sufficiently small $\epsilon>0$ such that
$Y\succ xx^{T}$ for all $(x,Y)$ satisfying (\ref{SDP:1}) and
\begin{equation}
\|x-x^*\|+\|Y-\widetilde{Y}\|<\epsilon.\label{eps}
\end{equation}
 By the optimality of $(x^*,\widetilde{Y})$, for all $(x,Y)$ satisfying (\ref{SDP:1}) and (\ref{eps}), we have
\begin{equation}
{\rm Tr}(Y)-2a_0^Tx \leq{\rm Tr}(\widetilde{Y})-2a_0^Tx^*. \label{Yxstar}
\end{equation}

Let $\widetilde{y}={\rm Tr}(\widetilde{Y})$. It is trivial to see that $(x^*,\widetilde{y})$ is a feasible solution of (LP), that is, all inequalities (\ref{LP-constraint}) hold true at $(x^*,\widetilde{y})$. Next we show that $(x^*,\widetilde{y})$ is also an optimal solution of (LP).

Define the following two continuous functions:
\begin{eqnarray}
&&f_1(x,y)=y-x^Tx,\nonumber\\
&&f_2(x,y)=\|x-x^*\|+\left\|xx^T+\frac{y-x^{T}x}{n}\cdot I-x^*x^{*T}-\frac{\widetilde{y}-x^{*T}x^*}{n}\cdot I\right\|.\nonumber
\end{eqnarray}
According to the above definitions of $x^*$, $\widetilde{Y}$ and $\widetilde{y}$, we have
\begin{eqnarray}
&&f_1(x^*,\widetilde{y})=\widetilde{y}-x^{*T}x^*={\rm Tr}(\widetilde{Y})-x^{*T}x^*>0,\nonumber\\
&& f_2(x^*,\widetilde{y})=0.\nonumber
\end{eqnarray}
Therefore, by the continuity of $f_1(x,y)$,  there exists $\delta_1>0$ such that
\begin{equation}
|f_1(x,y)-f_1(x^*,\widetilde{y})|<\frac{1}{2}f_1(x^*,\widetilde{y}) \label{f1}
\end{equation}
holds for any $(x,y)$ satisfying both (\ref{LP-constraint}) and $\|x-x^*\|+|y-\widetilde{y}|<\delta_1$. Notice that it follows from (\ref{f1}) that
\begin{equation}
f_1(x,y)>f_1(x^*,\widetilde{y})-\frac{1}{2}f_1(x^*,\widetilde{y})=\frac{1}{2}f_1(x^*,\widetilde{y})>0. \label{f11}
\end{equation}
Similarly, for the sufficient small $\epsilon>0$ used in (\ref{eps}), there exists $\delta_2>0$ such that  
\begin{equation}
|f_2(x,y)-f_2(x^*,\widetilde{y})|<\epsilon \label{f2}
\end{equation}
holds for any $(x,y)$ satisfying both (\ref{LP-constraint}) and $\|x-x^*\|+|y-\widetilde{y}|<\delta_2$.

Define $\delta=\min\{\delta_1,\delta_2\}>0$ and
\[
Y(x,y)=xx^{T}+\frac{y-x^{T}x}{n}\cdot I.
\]
Let $(x,y)$ be any vector satisfying both (\ref{LP-constraint}) and $\|x-x^*\|+|y-\widetilde{y}|<\delta$. Then both (\ref{f11}) and (\ref{f2}) hold true at $(x,y)$. It follows from (\ref{f11}) that $y>x^Tx$, which implies that
\[
Y(x,y)\succ xx^{T}.
\]
According to (\ref{f2}), we obtain
\begin{equation}
\|x-x^*\|+\left\|Y(x,y)- \widetilde{Y}\right\|=f_2(x,y)<\epsilon. \label{xY}
\end{equation}
Moreover, since $(x,y)$ satisfies (\ref{LP-constraint}) and ${\rm Tr}(Y(x,y))=y$, it holds that $(x,Y(x,y))$ satisfies (\ref{SDP:1}). Notice that (\ref{eps}) holds true at $(x,Y(x,y))$ due to (\ref{xY}). Therefore, (\ref{Yxstar}) holds at $(x,Y(x,y))$, that is,
\[
{\rm Tr}(Y(x,y))-2a_0^Tx \leq{\rm Tr}(\widetilde{Y})-2a_0^Tx^*.
\]
Since ${\rm Tr}(Y(x,y))=y$ and ${\rm Tr}(\widetilde{Y})=\widetilde{y}$, we now conclude that
 \begin{equation}
y-2a_0^Tx \leq \widetilde{y}-2a_0^Tx^*
\end{equation}
holds for any $(x,y)$ satisfying both (\ref{LP-constraint}) and $\|x-x^*\|+|y-\widetilde{y}|<\delta$.
It follows that $(x^*,\widetilde{y})$ is a local maximizer of (LP). Consequently,  $(x^*,\widetilde{y})$ is a (global) optimal solution of (LP) and hence
\[
v({\rm LP})=\widetilde{y}-2a_0^Tx^*={\rm Tr}(\widetilde{Y})-2a_0^Tx^*=v({\rm SDP}).
\]
The proof is complete.
\end{proof}

The assumption is necessary in Theorem \ref{thm:sdp1}. Suppose $v({\rm SDP})=v({\rm UQ})$, it may happen either $v({\rm LP})=+\infty$ or $v({\rm SDP})<v({\rm LP})<+\infty$, as demonstrated by the following example.
\begin{exam}\label{exam1}
Let $n=1$. For any $\alpha$, consider the following examples:
\begin{eqnarray*}
{\rm(UQ(\alpha))}~~&\max ~~& x^2-\alpha x\\
&{\rm s.t.}& x^2+x-4\leq 0,\\
&&x^2-x\leq 0.
\end{eqnarray*}
It always holds that $v({\rm SDP})=v({\rm UQ})$.
Moreover, one can verify that
\begin{itemize}
\item[(1)] If $\alpha\not\in[-1,1]$,  $v({\rm LP}(\alpha))=+\infty >v({\rm SDP}(\alpha))$. (Lemma \ref{lem:1})
\item[(2)] If $-1\leq \alpha < 1$,  $v({\rm LP}(\alpha))=2(1-\alpha) >v({\rm SDP}(\alpha)) =1-\alpha$.
\item[(3)] If $\alpha =1$,  $v({\rm LP}(\alpha))= v({\rm SDP}(\alpha)) =0$.
\end{itemize}
\end{exam}

Moreover, we can know more from \textit{any} optimal solution of the linear programming relaxation (LP).
\begin{thm}\label{thm:sdp2}
Suppose $v({\rm LP})< +\infty$ and let $(x^*, y^*)$ be  any optimal solution of ${\rm (LP)}$. Then, we have
\begin{itemize}
\item[(a)] If $x^{*T}x^*=y^*$, then $v({\rm LP})=v({\rm SDP})=v({\rm UQ})$.
\item[(b)]If $x^{*T}x^*> y^*$, then $v({\rm SDP})=v({\rm UQ})$.
\item[(c)] If $x^{*T}x^*<y^*$, then $v({\rm LP})=v({\rm SDP})$,
\end{itemize}
\end{thm}
\begin{proof}
Notice that we always have $v({\rm LP})\ge v({\rm SDP})\ge v({\rm UQ})$ according to their definitions.
In the case of $x^{*T}x^*= y^*$, the linear programming relaxation (LP) is tight, i.e., $v({\rm LP})=v({\rm UQ})$. Then it is trivial to see that $v({\rm LP})=v({\rm SDP})=v({\rm UQ})$.

Suppose $x^{*T}x^*> y^*$. Assume $v({\rm SDP})>v({\rm UQ})$, according to the proof of Theorem \ref{thm:sdp1},  there is an optimal solution of $({\rm SDP})$, denoted by
$(\widetilde{x},\widetilde{Y})$, such that ${\rm Tr}(\widetilde{Y})>\widetilde{x}^T\widetilde{x}$. And moreover, $(\widetilde{x},{\rm Tr}(\widetilde{Y}))$ is an optimal solution of $({\rm LP})$. Therefore, for any $\alpha \in [0,1]$, $(\alpha \widetilde{x}+(1-\alpha)x^*,\alpha{\rm Tr}(\widetilde{Y})+(1-\alpha)y^*)$ is also an optimal solution of $({\rm LP})$. Define
\[
F(\alpha)=\alpha{\rm Tr(\widetilde{Y})}+(1-\alpha)y^*-\|\alpha \widetilde{x}+(1-\alpha)x^*\|^2.
\]
Since $F(0)< 0<F(1)$ and $F(\alpha)$ is continuous, we conclude that there is an $\alpha \in (0,1)$ such that $F(\alpha)=0$, which implies that $v({\rm LP})=v({\rm UQ})$. It immediately follows that
$v({\rm SDP})=v({\rm UQ})$, which is a contradiction. Therefore, in this case, we must have $v({\rm SDP})=v({\rm UQ})$.

For the last case $x^{*T}x^*< y^*$, define $Y^*=x^*x^{*T}+\frac{y^*-x^{*T}x^*}{n}\cdot I$. Then, one can verify that $(x^*,Y^*)$ satisfies the constraints (\ref{SDP:1})-(\ref{SDP:2}) and the corresponding objective function value (\ref{SDP:0}) is
\[
\left(\begin{array}{cc}I&-a_{0}\\-a_{0}^T&0
\end{array}\right)\bullet \left(\begin{array}{cc}Y^*&x^*\\x^{*T}&1
\end{array}\right)=y^*-2a_0^Tx^*=v{\rm(LP)}.
\]
Since $v({\rm LP})\ge v({\rm SDP})$, $(x^*,Y^*)$ is an optimal solution to $({\rm SDP})$.
Thus, $v({\rm LP})=v({\rm SDP})$.
The proof is complete.
\end{proof}

Based on Lemma \ref{lem:1}, Theorems \ref{thm:sdp1} and \ref{thm:sdp2}, now we can prove Theorem \ref{thm:sdp0}.

\noindent\textbf{Proof of Theorem \ref{thm:sdp0}}.

Assume Case (i) holds, i.e., $a_0\not\in$ conv$\left\{a_1,\ldots,a_p\right\}$. According to Lemma \ref{lem:1}, $v({\rm LP})=+\infty$. Suppose $v({\rm SDP})=v({\rm UQ})$ is not true, then it must hold that $v({\rm SDP})>v({\rm UQ})$. According to Theorem \ref{thm:sdp1}, we have $v({\rm LP})=v({\rm SDP})<+\infty$, which is a contradiction.

Now we assume that Case (i) does not hold, but Case (ii) holds, i.e.,  there is a vector
$v\neq 0$ such that $(a_i-a_0)^Tv\geq 0$ for $i=1,\cdots,p$.

Based on the linear transformation $t= y-2a_0^Tx$, (LP) has the following equivalent reformulation:
\begin{eqnarray}
{\rm (LP')}~~&\max& t \nonumber\\
&{\rm s.t.}& t-2(a_i-a_0)^Tx+b_i\leq 0,~i=1,\ldots,p.\label{lpp}
\end{eqnarray}
According to Lemma \ref{lem:1}, $v({\rm LP'})<+\infty$. Then, $({\rm LP'})$ has an optimal solution, denoted by $(x^*,t^*)$. It is easy to see that any solution $(x, t^*)$ satisfying (\ref{lpp}) remains optimal for $({\rm LP'})$.  Define $\widetilde{x}(\alpha)=x^*+\alpha v$. For all $\alpha>0$, we have
\begin{eqnarray*}
t^*-2(a_i-a_0)^T\widetilde{x}(\alpha)+b_i &=&t^*-2(a_i-a_0)^Tx^*-2\alpha(a_i-a_0)^Tv+b_i\\
&\leq& t^*-2(a_i-a_0)^Tx^*+b_i\\
&\leq&0.
\end{eqnarray*}
That is, for any $\alpha>0$, $(\widetilde{x}(\alpha),t^*)$ is optimal for $({\rm LP'})$. Since $v\neq 0$, for sufficiently large $\alpha$, we have $t^*<\widetilde{x}(\alpha)^T\widetilde{x}(\alpha)$. Then, according to Theorem \ref{thm:sdp2}, we have $v({\rm SDP})=v({\rm UQ})$. This completes the proof of Theorem \ref{thm:sdp0}. \hfill \quad{$\Box$}\smallskip

\begin{rem}\label{rema}
Case (b) of Theorem \ref{thm:sdp2} implies that a solution of (UQ) can be obtained in polynomial time (e.g., via solving (SDP)), but (LP) itself  fails to find the global optimal solution of (UQ). It is not difficult to verify that, (SDP) is equivalent to the following second-order cone programming relaxation:
\begin{eqnarray}
{\rm (SOCP)}~~&\max& y-2a_0^Tx\label{SOCP}\\
&{\rm s.t.}& y-2a_i^Tx+b_i\leq 0,~i=1,\ldots,p,\nonumber\\
&&\left\|\left(\begin{array}{c}x\\ \frac{y-1}{2}\end{array}\right)\right\|
 \le  \frac{y+1}{2},\nonumber
\end{eqnarray}
where the last constraint  is actually equivalent to $x^Tx\leq y$.
\end{rem}

\subsection{Approximation algorithm for (${\rm UQ}$)}
For the hard case of (UQ) (i.e., Case (c) in Theorem \ref{thm:sdp2}), it often holds that $v({\rm UQ})<v({\rm LP})$. One can employ the approximation algorithms \cite{H15,Ts03} for solving the general nonconvex quadratic optimization with ellipsoid constraints $({\rm NC{\text-}EQP})$, which also provides an approximation bound of the returned solution.  In this section, we propose a new  approximation algorithm for (UQ) based on (LP) with an improved approximation bound, which is no longer dependent of $p$.

\begin{thm}\label{thm:4}
Suppose Assumption \ref{ass2} holds and
$v({\rm SDP})>v({\rm UQ})$.
We can find a feasible solution $x$ of (UQ) in polynomial time satisfying
\begin{eqnarray}
f_0(x)\geq \left(\frac{1-\gamma}{\sqrt{2}+\gamma}\right)^2\cdot v({\rm LP}),\label{inequa0}
\end{eqnarray}
where
$\gamma = \max_{i=1,\ldots,p}\frac{\|a_i\|}{\sqrt{-b_i+\|a_i\|^2}}<1$.
\end{thm}

\begin{proof}
Assumption \ref{ass2} implies that $v({\rm UQ})>0$. According to Theorem \ref{thm:sdp1}, (LP) has a bounded solution, denoted by  $(x^*, y^*)$. As $v({\rm SDP})>v({\rm UQ})$, it follows from Theorem \ref{thm:sdp2} that $x^{*T}x^*<y^*$.
Then there must exist a nonzero vector $t\in \Bbb R^n$ such that $x^{*T}x^*+t^Tt=y^*$.
Since $v({\rm LP})=y^*-2a_0^Tx^*$, for any $\beta$, it holds that
\begin{eqnarray}
(\beta^2+1)(x^{*T}x^*+t^Tt-2a_0^Tx^*)=(\beta^2+1)\cdot v({\rm LP}).\label{eq:11}
\end{eqnarray}
Based on the facts that $x^{*T}x^*-2a_0^Tx^*<y^*-2a_0^Tx^*=v({\rm LP})$ and $t^Tt>0$, there exits a real value $\beta > 0$ such that
\begin{eqnarray}
(x^*+\beta t)^T(x^*+\beta t)-2a_0^T(x^*+\beta t)= v({\rm LP}).\label{eq:22}
\end{eqnarray}
Define
\[
u_{1}=\frac{1}{\sqrt{1+\beta^2}}, ~u_{2}=\frac{\beta}{\sqrt{1+\beta^2}}.
\]
The equation (\ref{eq:22}) is equivalent to
\begin{eqnarray}
(u_1 x^*+u_2t)^T(u_1 x^*+u_2t)-2u_1a_0^T(u_1 x^*+u_2t)=u_1^2\cdot v({\rm LP}). \label{eq:222}
\end{eqnarray}
The equation (\ref{eq:11}) minus the equation (\ref{eq:22}) equals
\begin{eqnarray*}
(\beta x^*-t)^T(\beta x^*-t)-2\beta a_0^T(\beta x^*-t)=\beta^2\cdot v({\rm LP}),
\end{eqnarray*}
which further implies that 
\begin{eqnarray}
(u_2x^*-u_1t)^T(u_2x^*-u_1t)-2u_2a_0^T(u_2x^*-u_1t)=u_2^2\cdot v({\rm LP}).\label{eq:333}
\end{eqnarray}
Define
\[
s_{1}=u_1x^*+u_2 t,~s_{2}=u_2 x^*-u_1t.
\]
The equations (\ref{eq:222}) and (\ref{eq:333}) can be recast as
\begin{eqnarray}
&&s_{1}^Ts_{1}- 2u_1a_0^Ts_{1}=u_1^2\cdot v({\rm LP}) ,\label{n1}\\
&&s_{2}^Ts_{2}- 2u_2 a_0^Ts_{2}=u_2^2\cdot v({\rm LP}),\label{n2}
\end{eqnarray}
respectively. Then, we can verify that
\[
s_{1}^Ts_{1}+s_{2}^Ts_{2}=y^*,~ u_1s_1+u_2s_2=x^*.
\]
As $(x^*,y^*)$ satisfies the constraints (\ref{LP-constraint}) of (LP), we obtain
\begin{eqnarray}
s_{1}^Ts_{1}- 2u_1a_{i}^Ts_{1} +s_{2}^Ts_{2}- 2u_2 a_{i}^Ts_{2} \leq -b_i,~i=1,\ldots,p.\label{n3}
\end{eqnarray}
Define
\[
\rho_i=\frac{1}{\sqrt{-b_i+\|a_i\|^2}},~i=1,\ldots,p.
\]
It follows from Assumption \ref{ass2} that $0<\rho_i<\frac{1}{\|a_i\|}$.
We can equivalently rewrite the inequalities (\ref{n3}) as
\begin{eqnarray*}
\rho_i^2\|s_{1}-u_1a_i\|^2
+
\rho_i^2\|s_{2}-u_2  a_i\|^2\leq 1, ~i=1,\ldots,p.
\end{eqnarray*}
Then, we have
\begin{eqnarray*}
&\min& \left\{ \max_{i=1,\ldots,p}\frac{1}{u_1^2}\cdot \rho_i^2\|s_{1}-u_1a_i\|^2,
\max_{i=1,\ldots,p}\frac{1}{u_2^2}\cdot \rho_i^2\|s_{2}-u_2a_i\|^2\right\}\\
&\le&
\min  \left\{  \frac{1}{u_1^2},~\frac{1}{u_2^2}\right\}\\
&=&
\left\{\begin{array}{ll}1+\beta^2,
& {\rm if}~|\beta|\leq 1,\\
1+\frac{1}{\beta^2},
& {\rm otherwise},\end{array}\right.\\
&\leq&2.
\end{eqnarray*}
Consequently, there is an index $\overline{j}\in \{1,2\}$ such that
\begin{equation}
\rho_i\|s_{\overline{j}}/u_{\overline{j}}-a_{i}\|
\le \sqrt{2},~i=1,\ldots,p.\label{16}
\end{equation}

Next we will construct a feasible solution of (UQ) satisfying the inequality (\ref{inequa0}).
Define
\begin{eqnarray*}
\overline{x}&:=&\left\{\begin{array}{ll}s_{\overline{j}}/u_{\overline{j}},
& {\rm if}~a_{0}^Ts_{\overline{j}}/u_{\overline{j}}\leq 0,\\
-s_{\overline{j}}/u_{\overline{j}},
& {\rm otherwise},\end{array}\right.\\
\overline{\tau}&:=&\max\bigg\{\tau\in[0,1]:~f_i(\tau \overline{x})\le 0,~i=1,\ldots,p\bigg\}\\
&~=&\max\bigg\{\tau\in[0,1]:~\rho_i\|\tau\overline{x}-a_{i}\|\le 1,~i=1,\ldots,p\bigg\}.
\end{eqnarray*}
According to \eqref{16},
it holds that
\begin{eqnarray*}
\rho_i\|\overline{x}-a_{i}\|
&&\le~~ \max\left\{\rho_i\|s_{\overline{j}}/u_{\overline{j}}-a_{i}\|, \rho_i\|-(s_{\overline{j}}/u_{\overline{j}}-a_{i})-2a_{i}\| \right\}\\
&&\le~~ \max\left\{\sqrt{2}, \sqrt{2}+2\rho_i\|a_{i}\| \right\}\\
&&=~~\sqrt{2}+2\rho_i\|a_{i}\|.
\end{eqnarray*}
Therefore, for any $\tau\in[0,1]$, we obtain
\begin{eqnarray*}
\rho_i\|\tau\overline{x}-a_{i}\|
&&=~~\rho_i\|\tau(\overline{x}-a_{i})+(1-\tau)(-a_i)\|\\
&&\le~~ \tau\left(\sqrt{2}+2\rho_i\|a_{i}\|\right)
+(1-\tau)\rho_i\|a_{i}\|.
\end{eqnarray*}
Thus, for any $\tau\in\left[0,\frac{1-\rho_i\|a_{i}\|}{\sqrt{2}+\rho_i\|a_{i}\|}\right]$,
we have
\[
\rho_i\|\tau\overline{x}-a_{i}\|\leq 1,~i=1,\ldots,p,
\]
i.e.,
$\tau \bar{x}$ is feasible for (UQ).
According to the definition of $\overline{\tau}$, we obtain
\[
\overline{\tau}\geq \min_{i=1,\ldots,p}
\frac{1-\rho_i\|a_{i}\|}{\sqrt{2}+\rho_i\|a_{i}\|}=\frac{1-\max_{i=1,\ldots,p}\rho_i\|a_{i}\|}
{\sqrt{2}+\max_{i=1,\ldots,p}\rho_i\|a_{i}\|},
\]
where the equality holds true since $h(\gamma)=(1-\gamma)/(\sqrt{2}+\gamma)$ is a decreasing function over
$[0,1]$. Consequently, we have
\begin{align}
f_{0}(\overline{\tau}\overline{x}) =&\overline{\tau}^2\overline{x}^T\overline{x}-2\overline{\tau} a_{0}^T\overline{x}\nonumber\\
 \geq &\overline{\tau}^2\overline{x}^T\overline{x}-2\overline{\tau}^2 a_{0}^T\overline{x}\label{add0}\\
 \geq &\overline{\tau}^2\overline{x}^T\overline{x}-2\overline{\tau}^2 a_{0}^Ts_{\overline{j}}/u_{\overline{j}}\label{add1}\\
 =& \overline{\tau}^2({s_{\overline{j}}}^Ts_{\overline{j}}-2u_{\overline{j}}a_{0}^Ts_{\overline{j}})/{u_{\overline{j}}}^2\nonumber\\
 = &\overline{\tau}^2\cdot v({\rm LP})\label{add2},
\end{align}
where the inequality \eqref{add0} holds true since $a_0^T\overline{x}\le 0$ and $\overline{\tau}\geq \overline{\tau}^2$ (as $\overline{\tau}\in[0,1]$), the inequality \eqref{add1} is true as $a_0^T\overline{x}\leq a_{0}^Ts_{\overline{j}}/u_{\overline{j}}$,
and the equality \eqref{add2} follows from (\ref {n1}) and (\ref{n2}).
\end{proof}

\section{Standard quadratic programming  relaxation}
\subsection{A new and simple derivation}
The standard quadratic programming  relaxation for $({\rm CC_B})$ was first proposed by
Beck \cite{Be07}. First, by replacing the inner maximization problem with its Lagrangian dual (D-SDP), one obtains the following SDP problem:
\begin{eqnarray*}
({\rm SDP}(z))~~&\min ~~& y\\
&{\rm s.t.}& \left(\begin{array}{cc}(-1+\sum_{i=1}^p\lambda_{i})I &z-\sum_{i=1}^p\lambda_{i}a_{i}\\
z^T-\sum_{i=1}^p\lambda_{i}a_{i}^T&y+\sum_{i=1}^p\lambda_{i}(\|a_{i}\|^2-r_{i}^2)
\end{array}\right)\succeq 0,\\
&&\lambda_{i}\ge0,~i=1,\ldots,p.
\end{eqnarray*}
Then,  the minimax optimization problem $({\rm CC_B})$ is relaxed to a convex minimization problem:
\begin{eqnarray*}
({\rm DCC})~~&&\min_{z}~ v({\rm SDP}(z))+\|z\|^2 \\
&=&\min_{y,\lambda,z} ~ y+\|z\|^2\\
&&{\rm s.t.}~ \left(\begin{array}{cc}(-1+\sum_{i=1}^p\lambda_{i})I &z-\sum_{i=1}^p\lambda_{i}a_{i}\\
z^T-\sum_{i=1}^p\lambda_{i}a_{i}^T&y+\sum_{i=1}^p\lambda_{i}(\|a_{i}\|^2-r_{i}^2)
\end{array}\right)\succeq 0,\\
&&\lambda_{i}\ge 0,~i=1,\ldots,p.
\end{eqnarray*}
Based on a nontrivial analysis, Beck \cite{Be07}  showed that  (DCC) is further equivalent to the following standard convex quadratic programming problem (SQP):
\begin{eqnarray*}
({\rm SQP})~~&&\min_{\lambda} ~ \sum_{i=1}^p\lambda_{i}(r_{i}^2-\|a_{i}\|^2)+\left\|\sum_{i=1}^p\lambda_{i}a_{i}\right\|^2\\
&&{\rm s.t.}~
 \sum_{i=1}^p\lambda_{i}=1,~\lambda_{i}\ge 0,~i=1,\ldots,p.
\end{eqnarray*}
Let $\lambda^*$ be the optimal solution of (SQP). The approximate solution of ${\rm (CC_B)}$ is then given by
\begin{equation}
\overline{z}=\sum_{i=1}^p\lambda_i^*a_i \label{csol}
\end{equation}
so that we have
\begin{equation}
v({\rm SQP})= v({\rm SDP}(\overline{z}))+\|\overline{z}\|^2. \label{sqp:z}
\end{equation}
When $p\leq n$, according to Theorem \ref{Beck}, $v({\rm CC_B})=v({\rm DCC})=v({\rm SQP})$ and hence $\overline{z}$ is a global optimal solution of $({\rm CC_B})$ \cite{Be}.

In this section, we present a new and simpler derivation of (SQP).
We first  write the linear programming relaxation (as discussed in Section 3.2) of the inner maximization of (${\rm CC_B}$):
\begin{eqnarray*}
{\rm LP}(z)~~&\max& y-2z^Tx\\
&{\rm s.t.}&  y-2a_{i}^Tx+\|a_{i}\|^2\leq {r_{i}}^2,~i=1,\ldots,p,
\end{eqnarray*}
which is equivalent to its  dual linear programming:
\begin{eqnarray*}
({\rm LPD}(z))~~&&\min_{\lambda}\sum_{i=1}^p\lambda_i(r_i^2-\|a_i\|^2)\\
&&{\rm s.t.} \sum_{i=1}^p\lambda_ia_i=z,~\sum_{i=1}^p\lambda_i=1, ~\lambda_i\geq 0,~i=1,\ldots,p.
\end{eqnarray*}
Therefore, we immediately re-obtain (${\rm SQP}$) as follows:
\begin{eqnarray*}
v({\rm CC_B})&\le&\min_{z} \left\{v({\rm LP}(z))+\|z\|^2\right\}\\
&=&\min_{z}\left\{v({\rm LPD}(z))+\|z\|^2\right\}\\
&=&\min_{\lambda}\sum_{i=1}^p\lambda_i(r_i^2-\|a_i\|^2)+\left\|\sum_{i=1}^p\lambda_ia_i\right\|^2\\
&&{\rm s.t.} \sum_{i=1}^p\lambda_i=1, \lambda_i\geq 0,~i=1,\ldots,p.
\end{eqnarray*}
The above two different derivations imply the following interesting observation.
\begin{prop}\label{prop2}
Let $\overline{z}$ be  defined in (\ref{csol}). Then, we have
\begin{equation}
v({\rm SDP}(\overline{z}))=v({\rm LP}(\overline{z})).\label{lpsdp}
\end{equation}
\end{prop}
The following example shows that (\ref{lpsdp}) could be no longer true if $\overline{z}$ is replaced by any other $z$.
\begin{exam}\label{exam2}
Let $n=1$. Consider
\[
\min_z\max_{x^2+x-4\leq 0,x^2-x\leq 0}|x-z|^2.
\]
Solving the corresponding (SQP), we get $\overline{z}=\frac{1}{2}$.
According to Example \ref{exam1},
it is interesting to verify that
\[
v({\rm SDP}(z))<v({\rm LP}(z)),~\forall  z\neq \overline{z}.
\]
\end{exam}
%

\subsection{Approximation bound}
(SQP) is efficient to solve and  provides a candidate  solution ($\overline{z}$ (\ref{csol})) to $({\rm CC_B})$.
However, to the best of our knowledge, the worst-case quality of the returned solution $\overline{z}$ remains unknown.
In this section, we answer this question. Moreover, the
first approximation ratio between $v({\rm CC_B})$ and $v({\rm SQP})$ is established.

\begin{thm}
Suppose ${\rm int}(\Omega)\neq \emptyset$, for the returned solution $\overline{z}$ (\ref{csol}), we have
\begin{eqnarray}
v({\rm SQP})\ge \max_{x\in \Omega}\|x-\overline{z}\|^2 \ge v({\rm CC_B})
\ge
\left(\frac{1-\gamma}{\sqrt{2}+\gamma}\right)^2\cdot v({\rm SQP}),\label{ratio}
\end{eqnarray}
where $\gamma(<1)$ is equal to (or larger than) the optimal objective value of the following convex programming problem:
\begin{equation}
\min_{x\in\Bbb R^n}\max_{i=1,\ldots,p}\frac{\|x-a_{i}\|}{r_{i}}. \label{gam0}
\end{equation}
Moreover,  let $d_{\max}=\max_{i,j=1,\ldots,p}\|a_{i}-a_{j}\|$, $r_{\min}=\min_{i=1,\ldots,p}r_i$ and suppose $d_{\max} <\sqrt{2}~ r_{\min}$, then (\ref{ratio}) holds with
\begin{equation}
\gamma =\frac{d_{\max} }{\sqrt{2}~ r_{\min}}.
\label{gam1}
\end{equation}
\end{thm}
\begin{proof}
Let  $\overline{z}$ be defined in (\ref{csol}).
If $({\rm SDP}(\overline{z}))$ is tight for the inner maximization problem, then
$v({\rm CC_B})=v({\rm SQP})$ and there is nothing to prove.

We first suppose $0\in {\rm int}(\Omega)$.
According to Theorem \ref{thm:4},   we have
\begin{equation}
v({\rm LP}(z))\ge
 \max_{x\in \Omega}\{\|x\|^2-2x^Tz\}\ge \tau_0^2\cdot v({\rm LP}(z)),~\forall z\in\Bbb R^n,\label{z:1}
\end{equation}
where \[
\tau_0=\frac{1-\max_{i=1,\ldots,p}\frac{\|a_{i}\|}{r_{i}}}{\sqrt{2}+\max_{i=1,\ldots,p}\frac{\|a_{i}\|}
 {r_{i}}}.
 \]
Since $0\in {\rm int}(\Omega)$, we have $\|a_{i}\|<r_{i}$ for $i=1,\ldots,p$, which imply $\tau_0>0$. On the other hand, it is trivial to see $\tau_0<1$. Therefore, it holds that $\tau_0^2<1$ and then we obtain from (\ref{z:1}) that
\begin{eqnarray*}
\min_{z}\{v({\rm LP}(z))+\|z\|^2\}\ge \min_{z}\max_{x\in \Omega}\|x-z\|^2 \ge
\tau_0^2\min_{z}\{v({\rm LP}(z))+\|z\|^2\}.
\end{eqnarray*}
In view of the definition of $\overline{z}$, Proposition \ref{prop2}   and (\ref{sqp:z}), we have
\[
v({\rm SQP})= v({\rm LP}(\overline{z}))+\|\overline{z}\|^2 \ge v({\rm CC_B}) \ge
\tau_0^2 ( v({\rm LP}(\overline{z}))+\|\overline{z}\|^2) =\tau_0^2\cdot v({\rm SQP}).
\]
Besides, the first inequality of (\ref{ratio}) follows from Proposition \ref{prop2} and the definition of $({\rm SDP}(\overline{z}))$, that is,
\[
v({\rm LP}(\overline{z}))+\|\overline{z}\|^2=v({\rm SDP}(\overline{z}))+\|\overline{z}\|^2 \ge \|x\|^2-2\overline{z}^Tx+\|\overline{z}\|^2,~\forall~x\in\Omega.
\]

Now, let $x_0$ be any interior point of $\Omega$ since ${\rm int}(\Omega)\neq \emptyset$, i.e., $ \|x_0-a_{i}\|< r_{i}$, $i=1,\ldots, p$. 
Consider
\begin{eqnarray*}
({\rm CC_B'})~~\min_{z}\max_{\widetilde{x}\in \Omega(x_0)}\|\widetilde{x}-(z-x_0)\|^2,
\end{eqnarray*}
where
$\Omega(x_0)=\{\widetilde{x}\in \Bbb R^n:~ \|\widetilde{x}-(a_{i}-x_0)\|^2\leq r_{i}^2, ~i=1,\ldots, p\}$. Then it is trivial to see that $0\in {\rm int}(\Omega(x_0))$ and $v({\rm CC_B})=v({\rm CC_B'})$.

We write the linear programming relaxation for the inner maximization problem of $({\rm CC_B'})$ as follows:
\begin{eqnarray*}
{\rm LP'}(z)~~&\max& \widetilde{y}-2(z-x_0)^T\widetilde{x}\\
&{\rm s.t.}&  \widetilde{y}-2(a_{i}-x_0)^T\widetilde{x}+\|a_{i}-x_0\|^2\leq {r_{i}}^2,~i=1,\ldots,p.
\end{eqnarray*}
Let $(x^*,y^*)$ be an optimal solution of ${\rm LP}(z)$. Define
\[
\widetilde{x}:=x^*-x_0,~\widetilde{y}:=y^*-2x_0^Tx^*+x_0^Tx_0.
\]
We can verify  that $(\widetilde{x},\widetilde{y})$ is a feasible solution of ${\rm LP'}(z)$ and hence
\begin{equation}
v({\rm LP'}(z))\geq \widetilde{y}-2(z-x_0)^T\widetilde{x}
=v({\rm LP}(z))+\|z\|^2-\|z-x_0\|^2.\label{e:4}
\end{equation}
On the other hand,
let $(\widetilde{x}^*,\widetilde{y}^*)$ be an optimal solution of ${\rm LP'}(z)$. Define
\[
x:=\widetilde{x}^*+x_0,~y:=\widetilde{y}^*+2x_0^T\widetilde{x}^*+x_0^Tx_0.
\]
We can verify  that $(x,y)$ is a feasible solution of ${\rm LP}(z)$ and hence
\begin{equation}
v({\rm LP}(z))\geq y-2z^Tx
=v({\rm LP'}(z))-\|z\|^2+\|z-x_0\|^2.\label{e:5}
\end{equation}
Combining (\ref{e:4}) and (\ref{e:5}) yields
\[
v({\rm LP}(z))+\|z\|^2
=v({\rm LP'}(z))+\|z-x_0\|^2. 
\]
Therefore, we obtain
\begin{equation}
v({\rm SQP})=\min_z\{v({\rm LP}(z)+\|z\|^2\}=\min_z\{v({\rm LP'}(z)+\|z-x_0\|^2\}=v({\rm SQP'}).
\end{equation}

Since $0\in {\rm int}(\Omega(x_0))$, according to the first part of this proof, we have
\[
v({\rm SQP'})\ge \max_{\widetilde{x}\in \Omega(x_0)}\|\widetilde{x}-(\widetilde{z}-x_0)\|^2 \ge v({\rm CC_B'})\ge \tau(x_0)^2\cdot v({\rm SQP'}),
\]
where $\widetilde{z}=\sum_{i=1}^p\lambda_i^*(a_i-x_0)$, $\lambda^*$ is an optimal solution of $({\rm SQP'})$ and
\[
\tau(x_0)=\frac{1-\max_{i=1,\ldots,p}\frac{\|x_0-a_{i}\|}{r_{i}}}{\sqrt{2}+\max_{i=1,\ldots,p}\frac{\|x_0-a_{i}\|}
 {r_{i}}}.
\]
Therefore, it holds that
\begin{equation}
v({\rm SQP})\ge \max_{x\in \Omega}\|x-\widetilde{z}\|^2 \ge v({\rm CC_B})\ge \tau(x_0)^2\cdot v({\rm SQP}). \label{e:m}
\end{equation}

Since the inequality (\ref{e:m}) holds for arbitrary $x_0\in {\rm int}(\Omega)$, the best choice of  $x_0$ is to maximize the lower bound $\tau(x_0)$ in ${\rm int}(\Omega)$.  That is, we can select
\[
\tau(x_0) = \frac{1-\gamma}{\sqrt{2}+\gamma}
\]
with any $\gamma<1$ satisfying
\begin{eqnarray}
\gamma \ge \inf_{x_0\in {\rm int}(\Omega)}\max_{i=1,\ldots,p}\frac{\|x_0-a_{i}\|}{r_{i}}=\min_{x_0\in \Omega}\max_{i=1,\ldots,p}\frac{\|x_0-a_{i}\|}{r_{i}}.\label{gam}
\end{eqnarray}

Finally, we show that (\ref{gam1}) is a feasible choice of $\gamma$.
According to the definition of $r_{\min}$, we have
\begin{eqnarray*}
\min_{x_0\in \Omega}\max_{i=1,\ldots,p}\frac{\|x_0-a_{i}\|}{r_{i}} &\le& \min_{x_0\in \Bbb R^n}\max_{i=1,\ldots,p}\frac{\|x_0-a_{i}\|}{r_{\min}} \\
&=&
\frac{1}{r_{\min}}\cdot\min_{x_0\in \Bbb R^n}\max_{i=1,\ldots,p}\|x_0-a_{i}\|\\
&\le& \frac{1}{r_{\min}}\cdot\min_{x_0\in \Bbb R^n}\max_{a\in {\rm conv}\{a_{1},\cdots,a_{p}\}}\|x_0-a\|,
\end{eqnarray*}
where the last inequality holds since the optimal solution of the inner convex maximization problem in terms of $a$ is attained at one of the vertices $a_{1},\ldots,a_{p}$.
According to Example 3.3.6 in \cite{DB}, we have
\[
\min_{x_0\in \Bbb R^n}\max_{a\in {\rm conv}\{a_{1},\cdots,a_{p}\}}\|x_0-a\| \le d_{\max}\sqrt{\frac{n}{2(n+1)}}<\frac{d_{\max}}{\sqrt{2}}.
\]
Consequently, if $d_{\max} <\sqrt{2}~ r_{\min}$, $\gamma$ can be set as in (\ref{gam1}).
\end{proof}

\section{More polynomially solvable cases}
\subsection{Polynomially solvable cases for uniform quadratic optimization}
We first reformulate (UQ) as the following shifted version:
\begin{eqnarray*}
{\rm (UQ')}~~&\max& \|x-a_0\|^2-\|a_0\|^2\\
&{\rm s.t.}& \|x-a_0\|^2-2(a_i-a_0)^Tx +b_i-a_0^Ta_0\leq 0,~i=1,\ldots,p.
\end{eqnarray*}
The corresponding linear programming relaxation becomes
\begin{eqnarray*}
{\rm (LP')}~~&\max& z-\|a_0\|^2\\
&{\rm s.t.}& z-2(a_i-a_0)^Tx +b_i-a_0^Ta_0\leq 0,~i=1,\ldots,p.
\end{eqnarray*}
According to Remark \ref{rema}, the (SDP) relaxation is equivalent to the following (SOCP):
\begin{eqnarray}
{\rm (SOCP')}~~&\max& z-\|a_0\|^2\nonumber\\
&{\rm s.t.}& z-2(a_i-a_0)^Tx+b_i-a_0^Ta_0\leq 0,~i=1,\ldots,p,\nonumber\\
&& \left\|\left(\begin{array}{c}x-a_0\\ \frac{z-1}{2}\end{array}\right)\right\|\le  \frac{z+1}{2}.\label{con:socp}
\end{eqnarray}

We first solve $({\rm LP'})$. If $v({\rm LP'})=+\infty$, according to Theorem \ref{thm:sdp1}, $v({\rm SDP'})=v({\rm UQ'})$ and hence (UQ$'$) can be globally solved by (SOCP$'$). Otherwise, let $(x^*, z^*)$ be a returned optimal solution of $({\rm LP'})$. If $z^*=\|x^{*}-a_0\|^2$, $(x^*, z^*)$ is also an optimal solution of (UQ$'$).

Suppose $z^*< \|x^{*}-a_0\|^2$,
according to Theorem \ref{thm:sdp2}, $v({\rm SDP'})=v({\rm UQ'})$. Then, (UQ$'$) can be globally solved by (SOCP$'$).

Now, we assume that
\begin{equation}
z^*> \|x^{*}-a_0\|^2.\label{zx}
\end{equation}
Then we have the following reformulation. 
\begin{lem}\label{lem:2}
Suppose  $v({\rm LP'})<+\infty$. Let $(x^*, z^*)$ be an optimal solution of $({\rm LP'})$ and
 the assumption (\ref{zx}) holds. Then, (UQ$'$) is equivalent to the following nonconvex optimization problem:
\begin{eqnarray}
    {\rm (UQ'')}~~&\max& z-\|a_0\|^2\label{ineq:-1}\\
    &{\rm s.t.}& z-2(a_i-a_0)^Tx +b_i-a_0^Ta_0\leq 0,~i=1,\ldots,p,\label{ineq:0}\\
    &&z\leq \|x-a_0\|^2.\label{ineq:1}
\end{eqnarray}
\end{lem}
\begin{proof}
Since $v{\rm (UQ'')}\le v({\rm LP}')<+\infty$, and the feasible region of ${\rm (UQ'')}$ is nonempty and closed, ${\rm (UQ'')}$ has an attained optimal solution, denoted by $(x'',z'')$. If (\ref{ineq:1}) is active at $(x'',z'')$, then clearly $v{\rm (UQ'')}=v{\rm (UQ')}$. Otherwise,
we have $z''< \|x''-a_0\|^2$. It implies that $(x'',z'')$ is a local optimal solution of (\ref{ineq:-1})-(\ref{ineq:0}), which is also a local optimal solution of $({\rm LP}')$. Therefore, $(x'',z'')$ is an optimal solution of $({\rm LP}')$. It follows that $v({\rm LP}')=z''-\|a_0\|^2=z^*-\|a_0\|^2$, that is, $z''=z^*$. Define
\[
h(\alpha)=z^*-\|\alpha x''+(1-\alpha)x^*-a_0\|^2.
\]
Since $h(0)>0>h(1)$ and $h(\alpha)$ is continuous with respect to $\alpha$, there is an $\alpha^*\in(0,1)$ such that $h(\alpha^*)=0$ and $x(\alpha^*)=\alpha^* x''+(1-\alpha^*)x^*$ satisfies  (\ref{ineq:0}). Consequently, $(x(\alpha^*),z^*)$ is an optimal solution of  (UQ$''$). It follows that $v{\rm (UQ'')}=v{\rm (UQ')}$.
\end{proof}



\begin{lem}\label{lem:3}
Under the same assumption as in Lemma \ref{lem:2}, there is an optimal solution of (UQ$'$) at which (\ref{ineq:1}) is active and
   \begin{equation}
   {\rm rank}\left[a_j-a_0 ~(j\in J)\right]= n\label{rank}
   \end{equation}
holds with $J$ being the index set of all active constraints (\ref{ineq:0}).
\end{lem}
\begin{proof}
We first assume
\begin{equation}
v({\rm SDP'})>v({\rm UQ'}).\label{sduq}
\end{equation}
According to Lemma \ref{lem:2}, there is an optimal solution of ${\rm (UQ'')}$, denoted by $(\widetilde{x},\widetilde{z})$, such that $\widetilde{z}=\|\widetilde{x}-a_0\|^2$.
Let $J$ be the index set of active inequalities (\ref{ineq:0}) at $(\widetilde{x},\widetilde{z})$. Without loss of generality, we assume $J=\{1,\cdots,k\}$.
Suppose (\ref{rank}) does not hold, then
   \begin{equation}
   {\rm rank}\left[a_1-a_0,\cdots,a_k-a_0\right]\leq n-1.\label{rk}
   \end{equation}
There exists a vector $v\neq 0$ such that
    \[
    v^T(a_i-a_0)=0, ~i=1,\ldots,k.
    \]
We can further assume
\begin{equation}
v^T(\widetilde{x}-a_0)\ge 0,\label{v:1}
\end{equation}
otherwise, let $v:=-v$.
Then, for any sufficiently small $\epsilon>0$,   $(\widetilde{x}+\epsilon v,\widetilde{z})$ satisfies (\ref{ineq:0}) and the corresponding index set of active constraints is still $J$. Moreover, according to (\ref{v:1}), we have
\begin{equation}
\widetilde{z}=\|\widetilde{x}-a_0\|^2<\|\widetilde{x}-a_0\|^2+2\epsilon v^T(\widetilde{x}-a_0)+\epsilon^2 \|v\|^2
=\|\widetilde{x}+\epsilon v-a_0\|^2.\label{v:2}
\end{equation}
Therefore, $(\widetilde{x}+\epsilon v,\widetilde{z})$ remains an optimal solution of  (UQ$''$). Moreover, according to (\ref{v:2}),
$(\widetilde{x}+\epsilon v,\widetilde{z})$ is a local optimal solution of (\ref{ineq:-1})-(\ref{ineq:0}), which is also a local optimal solution of $({\rm LP}')$. Therefore, $(\widetilde{x}+\epsilon v,\widetilde{z})$ is also an optimal solution of $({\rm LP}')$. It follows from Theorem \ref{thm:sdp2} that $v({\rm SDP'})=v({\rm UQ'})$, which contradicts the assumption (\ref{sduq}).

Now, we consider the case that the assumption (\ref{sduq}) does not hold. Then, we have $v({\rm SOCP'})=v({\rm SDP'})=v({\rm UQ'})$. 
Notice that under assumption (\ref{zx}), Theorem \ref{thm:sdp2} implies that $v({\rm SDP'})=v({\rm LP'})$. It turns out that  $(x^*,z^*)$, the optimal solution of $({\rm LP'})$, remains an optimal solution of  $({\rm SOCP'})$. For any other (if exists) optimal solution of $({\rm SOCP'})$, denoted by $(\widetilde{x},\widetilde{z})$, we have
\[
v({\rm SOCP'})=z^*-\|a_0\|^2=\widetilde{z}-\|a_0\|^2.
\]
Then, $z^*=\widetilde{z}$ and the set of optimal solutions of $({\rm SOCP'})$ is characterized as
\begin{equation}
\{(x,z^*): z^*-2(a_i-a_0)^Tx+b_i-a_0^Ta_0\leq 0,~i=1,\ldots,p,~ \|x-a_0\|^2 \le z^*\}.
\end{equation}
It follows that  $v({\rm SOCP'})=v({\rm UQ'})$ if and only if
\begin{eqnarray}
z^*= &\max&\|x-a_0\|^2 \label{max:1}\\
&{\rm s.t.}& z^*-2(a_i-a_0)^Tx+b_i-a_0^Ta_0\leq 0,~i=1,\ldots,p.\label{max:2}
\end{eqnarray}
Since the problem (\ref{max:1})-(\ref{max:2}) is a strictly convex maximization over a polytope, its optimal solution must be an extreme point, i.e., there are $n$ active constraints in (\ref{max:2})  such that their coefficient vectors in terms of $x$ are linearly independent. The proof is complete.
\end{proof}

According to Lemma \ref{lem:3}, it is sufficient to enumerate all sets of $n$ indices  of the constraints (\ref{ineq:0}) with   linearly independent coefficient vectors.
For each such selection, letting the $n$ constraints be active yields the linear equations
\[
ze-Ax+b-(a_0^Ta_0)e=0,
\]
where $e=(1,\cdots,1)^T\in\Bbb R^n$, $A$ is an $n\times n$ matrix composed by the $n$ linearly
independent coefficient vectors and $b$ is the corresponding $n$-dimensional vector $(b_i)$.  Since $A$  is invertible, we have
\begin{equation}
x(z)=A^{-1}(b-(a_0^Ta_0)e+ze).\label{xA}
\end{equation}
Substituting (\ref{xA}) into $z=\|x-a_0\|^2$, as (\ref{ineq:1}) is active, we obtain a univariate quadratic equation in terms of $z$:
\[
z= \|A^{-1}b-(a_0^Ta_0)A^{-1}e+zA^{-1}e-a_0\|^2,
\]
which has at most two explicit real solutions $z_1\ge z_2$. If $x(z_1)$ (\ref{xA}) satisfies all the other inactive constraints (\ref{ineq:1}), then $x(z_1)$ is a candidate optimal solution of (UQ$'$). Otherwise, if $x(z_2)$ (\ref{xA}) is feasible, add $x(z_2)$ to the candidate solution list. Suppose neither $x(z_1)$ nor $x(z_2)$ is feasible, we conclude that the current selection of the $n$ linearly independent constraints cannot provide an optimal solution of (UQ$'$). Totally, there are at most $\left(p\atop n\right)=\frac{p!}{n!(p-n)!}$ candidate solutions and $x(z)$ with the maximal $z$ is the exact optimal solution of (UQ$'$).

As a summary, we have the following result.


\begin{thm}\label{pol}
The worst-case complexity for solving the problem $({\rm UQ})$ is
\begin{equation}
T(n,p):=
{\rm LP}(n+1, p)+\max\left\{{\rm SOCP}(n+1, p), \left(p\atop n\right)\cdot O(n^3)\right\},\label{cpx}
\end{equation}
where ${\rm LP}(n+1, p)$ is the time complexity for solving a linear programming problem in $(n+1)$-dimension with $p$ constraints and ${\rm SOCP}(n+1, p)$ is the time for solving a second-order cone programming problem in $(n+1)$-dimension with one second-order conic constraint and $p$ linear constraints.
\end{thm}

It is interesting to observe that, since (\ref{sduq}) already implies (\ref{zx}), there is actually no need to make the assumption (\ref{zx}) in Lemma \ref{lem:2}. Thus, the enumeration procedure itself is sufficient to find the optimal solution of (UQ$'$). That is, in the worst case, Theorem \ref{pol} could be improved as follows.
\begin{thm}\label{pol2}
Suppose either $n$ is fixed or $p=n+q$ with a fixed integer $q$, then $({\rm UQ})$ is strongly polynomially solvable in at most $\left(p\atop n\right)\cdot O(n^3)$ time.
\end{thm}

Finally, if the optimal solution of (UQ$'$) is not unique,  all the extreme points of the optimal solution set  are collected by the enumeration procedure.

\subsection{Polynomially solvable cases for ${\rm (CC_B)}$}
Notice that ${\rm (CC_B)}$ (\ref{ccb}) can be reformulated as
\begin{equation}
\min_{z}\left\{ f(z):=z^Tz+\max_{x\in \Omega} \{x^Tx-2z^Tx\}\right\}.\label{ccb2}
\end{equation}
Notice that $f(z)$ is strictly convex and nonsmooth.
Hence, ${\rm (CC_B)}$ is polynomially solved with the ellipsoid method \cite{N2004}, if the subgradient of $f(z)$ can be obtained in polynomial  time.

For the sake of completeness, we briefly present the ellipsoid method for solving the nonsmooth convex optimization:
\begin{eqnarray}\label{Q:1}
f^*=\min\left\{f(x):~x\in Q:=\{x\in\Bbb R^n:~\overline{f}(x)\leq 0\} \right\},
\end{eqnarray}
where $f(x)$ and $\overline{f}(x)$ are convex (possibly nonsmooth) functions, $Q$ is a bounded closed convex set with nonempty interior.
Let $g(x)$ and $\overline{g}(x)$ be the subgradients of $f(x)$ and $\overline{f}(x)$, respectively.
\begin{center}
\fbox{\shortstack[l]{
{\bf Ellipsoid method \cite{N2004}}\\
1.~Initialize $y_0\in \Bbb R^n$. Choose $R>0$ such that $\|y-y_0\|\le R$ for all $y\in Q$.\\
~~~~Set $H_0=R^2\cdot I$ and $k=0$. \\
2.~Repeat until a stopping criterion is reached:\\
~~~~$g_k=\left\{
 \begin{array}{l}
 g(y_k),~{\rm if}~y_k\in Q,\\
 \overline{g}(y_k),~{\rm if}~y_k\notin Q,
 \end{array}
 \right.$\\
~~~~$y_{k+1}=y_k-\frac{1}{n+1}\frac{H_kg_k}{\sqrt{g_k^TH_kg_k}}$,~
$H_{k+1}=\frac{n^2}{n^2-1}\left(H_k-\frac{2}{n+1}\frac{H_kg_kg_k^TH_k}{g_k^TH_kg_k}\right)$,\\
~~~~$k:=k+1$.
}}
\end{center}

\begin{thm}[Theorem 3.2.8, \cite{N2004}]\label{thm:elip}
Let $f(x)$ (\ref{Q:1}) be Lipschitz continuous on $\{x\in\Bbb R^n:\|x-x^*\|\le R\}$ with some constant $M$. Assume that there exists some $\rho > 0$ and $\overline{x}\in Q$ such that  $\{x\in\Bbb R^n:\|x-\overline{x}\|\le \rho\}\subseteq Q$, then for
\begin{equation}
k>2(n+1)^2\log(R/\rho),\label{ep:1}
\end{equation}
we have $Q\cap \{y_0,y_1,\ldots,y_{k-1}\}\neq \emptyset$
and
\begin{equation}
\min_{0\le j\le k,y_j\in Q}f(y_j)-f^*\leq \frac{1}{\rho}MR^2\cdot e^{-\frac{k}{2(n+1)^2}}.\label{ep:2}
\end{equation}
\end{thm}

Now, we employ the above ellipsoid method \cite{N2004} to solve ${\rm(CC_B)}$ (\ref{ccb2}). Let $z^*$ be the optimal solution of ${\rm(CC_B)}$. As it is well-known that $z^*\in \Omega$, we can rewrite  (\ref{ccb2}) in  the formulation as in (\ref{Q:1}),
where
\[
Q=\{z\in\Bbb R^n:\|z-a_1\|\leq r_1\}~{\rm and}~\overline{f}(z)= \|z-a_1\|^2-r_1^2.
\]
Since $Q$ itself is a ball, we have $\rho=R=r_1$. Notice that $\overline{f}(z)$ is a quadratic function and hence differentiable. Now, we compute the subgradient of $f(z).$   Let $x^*(z)={\rm arg}\max_{x\in \Omega}\{x^Tx-2z^Tx\}$ in evaluating $f(z)$ (\ref{ccb2}). It is not difficult to verify that
\[
g(z)=2(z-x^*(z))
\]
is a subgradient of $f(z)$ at $z$, i.e.,
\[
f(z_2)\geq f(z_1)+g(z_1)^T(z_2-z_1),~\forall z_1,z_2\in Q.
\]
Therefore, we have
\begin{eqnarray*}
f(z_1)-f(z_2)&\leq &g(z_1)^T(z_1-z_2)\\
&\leq&\|g(z_1)\|\cdot\|z_1-z_2\|\\
&\leq&2(\|z_1\|+\|x^*(z_1)\|)\|z_1-z_2\|\\
&\leq&4(\|a_1\|+r_1)\|z_1-z_2\|.
\end{eqnarray*}
The above  inequalities still hold if $z_1$ and $z_2$ are exchanged. Thus,
$f(z)$ is Lipschitz continuous on $Q$ with a constant $M=4(\|a_1\|+r_1)$.

For our case ${\rm (CC_B)}$, the assumptions of Theorem \ref{thm:elip} are satisfied. Moreover, the conclusions (\ref{ep:1}) and (\ref{ep:2}) are reduced to $k>0$ (as $\rho=R$) and
\begin{equation}
\min_{0\le j\le k,y_j\in Q}f(y_j)-f^*\leq  4(\|a_1\|+r_1)r_1\cdot e^{-\frac{k}{2(n+1)^2}},
\label{ep:3}
\end{equation}
respectively.
According to (\ref{ep:3}) and Theorem \ref{pol}, we have the following complexity result.
\begin{cor}
Suppose either $n$ is fixed or $p=n+q$ with a fixed integer $q$.
An $\epsilon$-approximate solution of ${\rm (CC_B)}$ (i.e., a vector $\widetilde{z}$ satisfying  $f(\widetilde{z})\le v{\rm (CC_B)}+\epsilon$) can be found in polynomial time. More precisely, the complexity is at most
\[
T(n,p)\cdot O(n^2)\log\frac{4(\|a_1\|+r_1)r_1}{\epsilon},
\]
where $T(n,p)$ is defined in (\ref{cpx}).
\end{cor}


\section{Conclusion}
Finding the smallest ball enclosing the intersection of $p$ given balls in dimension $n$, denoted by ${\rm(CC_B)}$, is a classical mathematical problem. From the view of optimization, ${\rm(CC_B)}$ is a minimax problem and the inner maximization problem is a nonconvex uniform quadratic optimization (UQ). It is the first time to show  ${\rm(CC_B)}$ is NP-hard, though ${\rm(CC_B)}$ in the plane is efficiently and strongly polynomially solved.  It is known  that when $p\le n$ (UQ) enjoys the strong duality. In this paper, we present a simple linear programming ${\rm(LP)}$ relaxation for (UQ), which leads to a more general sufficient condition for the strong duality of (UQ). Based on ${\rm(LP)}$, we propose a simple derivation of the standard convex quadratic programming (${\rm SQP}$) relaxation for ${\rm(CC_B)}$. Strong duality of (UQ) implies that (${\rm SQP}$) is tight when $p\le n$. However, this is not true when $p>n$. Generally, the first approximation bound of the solution obtained by (${\rm SQP}$) is established, which is independent of the number $p$.  Finally, with the help of  ${\rm(LP)}$, we show the polynomial solvability of ${\rm(CC_B)}$ under the  assumption either $n$ or $p-n>0$ is fixed. However, it is not known whether ${\rm(CC_B)}$ is \emph{strongly}  polynomially solvable even when  $n=3$.

\section*{Acknowledgments}
The authors are grateful to the two anonymous referees for very helpful comments and suggestions.

\end{document}